\documentclass[12pt]{article}
\usepackage{amsfonts,amsmath,amssymb,amsthm,color,tikz,adjustbox}
\usepackage{amsfonts,url}
\usepackage[T1]{fontenc}
\usepackage{mathrsfs,eucal}
\usepackage{dsfont}
\usepackage[utf8]{inputenc}
\usepackage{lmodern}
\usepackage{graphicx}

\usepackage{pythonhighlight}

\usepackage[colorlinks=true, urlcolor=blue,linkcolor=blue, citecolor=blue]{hyperref}

\usepackage{graphics,graphicx,amsmath}

\newcommand{\R}{\mathbb{R}}

\newcommand{\eps}{\varepsilon}
\newcommand{\egf}[1]{\exp\left(\e^{#1} - 1\right)}
\DeclareMathOperator\Log{Log}

\DeclareMathOperator\erfc{erfc}
\DeclareMathOperator{\e}{e}

\newtheorem{prop}{Proposition}[section]
\newtheorem{lemma}[prop]{Lemma}

\newtheorem{corollary}[prop]{Corollary}
\newtheorem{theorem}[prop]{Theorem}
\newtheorem{remark}[prop]{Remark}

\newtheorem{fact}[prop]{Proposition}

\def\({\left(}
\def\){\right)}

\def\[{\left[}
\def\]{\right]}

\allowdisplaybreaks

\numberwithin{equation}{section}

\begin{document}

\title{
\huge
Explicit bounds for Bell numbers \\and their ratios
} 

\author{
Jerzy Grunwald\thanks{Faculty of Pure and Applied Mathematics, Wroc{\l}aw University of Science and Technology, Ul. Wybrze\.ze Wyspia\'nskiego 27, Wroc{\l}aw, Poland.
e-mail: {\tt grunwaldjerzy@gmail.com}}
\and
Grzegorz Serafin\thanks{Faculty of Pure and Applied Mathematics, Wroc{\l}aw University of Science and Technology, Ul. Wybrze\.ze Wyspia\'nskiego 27, Wroc{\l}aw, Poland.
e-mail: {\tt grzegorz.serafin@pwr.edu.pl}}
}

\maketitle

\begin{abstract} 
In this article, we provide a comprehensive analysis of the asymptotic behavior of Bell numbers, enhancing and unifying various results previously dispersed in the literature. We establish several explicit lower and upper bounds. The main results correspond to two  asymptotic forms expressed by means of the Lambert $W$ function. As an application,  some straightforward elementary bounds are derived.  Additionally, an absolute convergence rate of the ratio of the consecutive Bell numbers is derived.  The main challenge was to obtain satisfactory  constants, as the Bell numbers  grow rapidly, while  the convergence rates are rather slow. 

\end{abstract}

\section{Introduction}

The Bell numbers $B_n$ are classical objects in combinatorial theory and have been studied for more than  one and a half centuries. They are subject of numerous studies as they have various applications in other areas of mathematics. In particular, $B_n$ represents the $n$'th moment of the Poisson distribution with the intensity parameter equal to $1$. In spite  of   many formulae and  relationships involving Bell numbers, there is known no simple explicit formula that is applicable for evaluating them for large $n$. The estimates are therefore strongly desirable. The first remarkable bound is due to de Bruijn  \cite{B}, who showed 
\begin{align}\label{eq:Brujin}
\frac{\ln B_n}{n}=\(\ln n-\ln\ln n-1+\frac{\ln\ln n}{\ln n}+\frac1{\ln n}+\frac12\(\frac{\ln\ln n}{\ln n}\)^2\)+O\(\frac{\ln\ln n}{(\ln n)^2}\).
\end{align}
Nevertheless, one can see that this result is imprecise for two reasons: the constants are unknown and the multiplicative error term $\exp\({n\ln\ln n/(\ln n)^2}\)$ tends to infinity as $n\rightarrow\infty$. The asymptotics of $B_n$ have been derived in \cite{L} (see also \cite{GKP, O}) in terms of  the Lambert $W$ function, which is the inverse of the function $[-1,\infty)\ni x\rightarrow x\e^x$:
\begin{align}\label{eq:lovasz}
\frac{ B_n}{E_n^*}\longrightarrow 1,\ \ \ \ \ \text{ as } n\rightarrow \infty,
\end{align}
where 
$$E_n^*=\frac{\exp\({\e^{W(n)}+nW(n)-(n+1)}\)}{ \sqrt{1+W(n)} }.$$
The most precise bounds are due to D. Knuth, who showed in \cite[formulae (30) and (31) in Section 7.2.1.5.]{K} that 
\begin{align}\label{eq:K1}
B_n&=E_n^*\(1+\tfrac{1}{12\e^{W(n)}}+O(\tfrac1n)\),\\\label{eq:K2}
\frac{B_n}{B_{n-1}}&=\e^{W(n)}+O(\tfrac1{W(n)}).
\end{align}
In fact, he claimed that the error bound in \eqref{eq:K1} is $O((\ln n)/n)^2$, which seems to be a mistake. This might be verified by following the argumentation in \cite{K}, and is confirmed by Theorem \ref{thm:master} of this article. It is interesting, despite  their quality, the above bounds seem to be often overlooked in the context of the Bell numbers. 
 The convergence rate to another asymptotic form 
 $$E_n=\frac{n!\egf{R}}{R^{n} \sqrt{2\pi(n+1) (R + 1)}}, \qquad R = W(n+1).$$
 has been derived in \cite[Proposition VIII.3.]{FS}
\begin{align}
\label{eq:FS}
\frac{B_n}{E_n}=1+O\big(\e^{-W(n+1)/5}\big).
\end{align}
However, the error bound turns out not to be optimal.

All the aforementioned results have been achieved using analytic tools. Recently, some probabilistic approaches have been proposed as well. For instance, the asymptotics have been rediscovered in \cite{TE}. Furthermore, an interesting  argument  allowed the authors of  \cite{AAR}    to obtain the bound $B_n\leq \sqrt{1+W(n)}E_n^*$. It does not exactly match the  asymptotic form, but its proof is extremely  short, which deserves some attention.

The main  weakness of the results presented above is that one cannot use them to conclude any bound for $B_n$ for a given $n$.  The explicit upper  bound  
\begin{align}\label{eq:PMS}
B_n\leq \(\frac{0.792n}{\ln(n+1)}\)^n,\hspace{40pt}n\geq1,
\end{align}
is derived in \cite{BT}, however, it does not recover the asymptotic behavior of the Bell numbers. Note that the remarkable interest in the article \cite{BT} reveals the requirement for bounds of that type.

In this paper we improve the  existing results in several directions. First of all, we establish two-sided estimates of $B_n$, which describe asymptotic behaviour  of the Bell numbers with precise rates of convergence (see Theorem \ref{prop:main2}):  
	\begin{align}\label{eq:mainintro}
		1-\frac15\frac{\ln n}{n}\leq \frac{B_n}{E_n^*}\leq 1, \hspace{40pt}n\geq2.
\end{align}
Additionally, in Corollary \ref{cor:errorlim} we show that the order of the above error term is optimal. In Theorem \ref{thm:master} an even more precise bound is derived, which provides explicit constants in the bound \eqref{eq:K1}.  Nevertheless, it is  more complex and  plays rather an auxiliary  role. Many results are presented in the language of the asymptotic form $E_n$ as well, however, the one above is the most elegant one. Furthermore, 
the exponent appearing in the definition of $E^*_n$ is equal to $\int_0^nW(x)dx$, which appears to be quite helpful in some of the proofs.
To the authors knowledge, this simplification has not been  observed in the context of Bell numbers  so far. As an application of the bound \eqref{eq:mainintro} we derive in Proposition \ref{prop:ele} the following simple elementary bounds
$$
	\(\frac1{\e}\frac{n}{\ln n}\)^n\leq B_n
	\leq \(\frac34\frac{n}{\ln n}\)^n,\hspace{40pt}n\geq2.
$$
Here, the main contribution is the lower bound, as  the upper one is very similar to  \eqref{eq:PMS}.
 Another main result of the paper is the following bound of the ratio of two consecutive Bell numbers (see Theorem \ref{thm:B/B})
$$\left|\frac{B_{n}}{B_{n-1}}-\e^{W(n)}\right|\leq  \frac87\frac1{W(n)}, \hspace{40pt}n\geq1,$$ 
which complements \eqref{eq:K2} with explicit constants.
Let us point out that  even though the function $\e^{W(x)}$ (which is the inverse of $x\ln x$) tends to infinity as $x\rightarrow\infty$, the obtained error is an absolute one.  

The basic idea of the proofs  relies on the saddle point method, that  was already  used in e.g., \cite{B, FS, K, MW}.  these references are books that consider the Bell numbers as just one of many examples, with relatively little attention devoted to them, which was one of the motivations to write  this article.  Although the starting point of the main proof is not new, many obstacles appear when it comes to the details. This is a consequence of the fact that some of the  expressions in the proofs converge very slowly. Some difficulties could be avoided by considering very large $n$, but then verifying the initial values might be beyond  the capabilities of current computers, as the Bell numbers grow very rapidly. We therefore propose an approach that ensures some kind of balance between optimization of constants and effort put in. In particular, properties of the Lambert $W$ functions are intensively exploited. After all, we leave for numerical verification  the obtained bounds for $n\leq 741$, which is equivalent  to the condition $W(n+1)\geq5$. This  is possible to perform quickly on an average computer. Additionally, the crucial bounds are customisable, so that one can easily improve the constants if interested in larger indices.

The paper is organized as follows. In Section 2 the Bell numbers and the Lambert $W$ function are introduced, and properties of some special sequence are studied, that are frequently used later on.  Section 3 contains customisable bounds on integrals constituting the numbers $B_n$. Section 4 is devoted to explicit convergence rates of $B_n$, and in Section 5 we deal with the ratio $B_n/B_{n-1}$. In Appendix, one can find a code in Python  of a program that verifies numerically some of the results for the initial Bell numbers.

\section{Preliminaries}

\subsection{Bell numbers}
The Bell numbers are  numbers of partitions of a set of $n$ elements. 
They satisfy the following recurrence formula
$$B_0=1,\hspace{10mm}B_{n+1}=\sum_{k=1}^n{n\choose k}B_k,$$
which is usually exploit to produce consecutive values of $B_n$. One can also use it to obtain the  exponential generating function 
\begin{align}\label{eq:egf}
B(t):=\sum_{n=0}^\infty  B_n\frac{t^n}{n!}=\e^{\e^t-1},\ \ \ t\in\R.
\end{align}
The famous Dobi{\'n}ski formula \cite{D} states that 
\begin{align}\label{eq:Dobinski}
B_n=\frac1{\e}\sum_{k=0}^\infty\frac{k^n}{k!},\ \ \ n\geq1.
\end{align}
As mentioned in Introduction, the main two asymptotic forms of the Bell numbers are
\begin{align}\label{eq:E_n}
E_n&=\frac{n!\egf{R}}{R^{n} \sqrt{2\pi(n+1) (R + 1)}}, \quad R = W(n+1)\\[3pt]\label{eq:E_n^*}
E_n^*&=\frac{\exp\({\e^{W(n)}+nW(n)-(n+1)}\)}{ \sqrt{1+W(n)} },
\end{align}
i.e., we have
$$\lim_{n\rightarrow\infty}\frac{B_n}{E_n}=\lim_{n\rightarrow\infty}\frac{B_n}{E_n^*}=1.$$
\subsection{Lambert $W$ function}
By $W(x)$, $x\geq0$, we denote the principal branch of the Lambert $W$ function, which is the inverse of the function $[-1,\infty)\ni x\rightarrow x\e^x$. Directly from the definition we obtain the relation
\begin{align}\label{eq:eW=x/W}
\e^{W(x)}=\frac x{W(x)}.
\end{align}
Using elementary calculus we get the following  formula for the derivative of $W(x)$
\begin{align}\label{eq:derW}
W'(x)=\frac{1}{x+{\e^{W(x)}}}.
\end{align}
In particular, this means that  $W$ is concave. Thus, for $0<x<y$ we have
\begin{align}\label{eq:aux7}
W(y)-W(x)&\leq W'(x)(y-x)=\frac{y-x}{x+\e^{W(x)}}\leq \frac{y-x}{x}.
\end{align}

The function $W$ cannot be expressed in terms of elementary functions, but several series representations are available. One of them is given by \cite{CGHJK}
$$W(x)=\ln x-\ln\ln x+\sum_{k=0}^\infty\sum_{m=1}^{\infty}\frac{(-1)^k|s(k+m,k+1)|}{m!}\frac{(\ln\ln x)^m}{(\ln x)^{k+m}},$$
where $s(n,k)$ stands for the Stirling numbers of the first kind. Nevertheless, the above series is not easy to work with, thus rather  simple bounds of the function $W(x)$ are required. One can easily check that  
 \begin{align}\label{eq:Wbounds0}
\ln x-\ln\ln x\leq W(x)\leq \ln x,\ \ \ x\geq \e.
\end{align}
%

The value of the Lambert $W$ function at $1$  is called the omega constant and equals
\begin{align}\label{eq:omega}\Omega=W(1)=0.567\ldots.
\end{align}

At the end of this section we invoke the formula for the integral of the Lambert $W$ function. Namely, by virtue of  the standard formula for the integral of an inverse function, we have
\begin{align}\label{eq:intW}
\int_0^xW(s)ds=\e^{W(x)}+xW(x)-x-1, \hspace{20mm} x\geq0.
\end{align}

\subsection{A special sequence}
 Let us denote 
$$q_n:=Q(W(n+1)):=1-\e^{-W(n+1)}\frac{1-\tfrac3{2W(n+1)}-\tfrac{10}{{W(n+1)}^2}-\tfrac{9}{{W(n+1)}^3}+\tfrac1{{W(n+1)}^4}}{12\(1+\tfrac1{W(n+1)}\)^{3}},$$
which  turns out to play an important role in asymptotics of the Bell numbers. In particluar, it approximates $B_n/E_n$ in Theorem \ref{thm:master}. Here are some properties of the sequence $q_n$, that will be used in the sequel.
\begin{lemma}
If $W(n+1)\geq5 $ (i.e., $n\geq 742$), then 
\begin{align}\label{eq:q0}
1\geq q_n&\geq 1-\frac{\e^{-W(n+1)}}{12},\\[5pt]\label{eq:q1}
q_n+1.6\e^{-2W(n+1)}&\leq 1,\\[2pt]\label{eq:q2}
|q_{n+1}-q_n|&\leq \frac{\e^{-W(n+1)}}{10(n+1)}.
\end{align}
\end{lemma}
\begin{proof}Regarding \eqref{eq:q0}, it suffices to show that the expression 
$$\tfrac3{2x}+\tfrac{10}{{x}^2}+\tfrac{9}{{x}^3}-\tfrac1{{x}^4}$$ 
is positive and less than one for $x\geq 5$. Positivity is clear. Furthermore, it is a decreasing function of $x$ (the negative sign of $1/x^4$ does not change much),  and its value at $5$ is less than one, which proves \eqref{eq:q0}.
  
Next,  we have
$$Q(x)+1.6\e^{-2x}=1-\frac{\e^{-x}}{12(1+\frac1x)^3}\big(\underbrace{1-\tfrac3{2x}-\tfrac{10}{{x}^2}-\tfrac{9}{{x}^3}+\tfrac1{{x}^4}-19.2(1+\tfrac1x)^3\e^{-x}}_{:=f(x)}\big).$$
Since the function $f$ is  increasing on $[5,\infty)$  and  
 
$$f(5)=1-\tfrac3{10}-\tfrac{2}{5}-\tfrac{9}{125}-\tfrac1{625}-19.2(\tfrac65)^3\e^{-5}=\frac{141.5-20736\e^{-5}}{625}\geq0,$$
we get  \eqref{eq:q1}.

Finally, we will derive \eqref{eq:q2} bounding  the derivative of $Q$. Precisely, we have
\begin{align*}
Q'(x)&=\frac{\e^{-x}}{24}\frac{2x^6-x^5-32x^4-72x^3-50x^2+10x+2}{x^2(x+1)^4}\\
&=\frac{\e^{-x}}{24} \left[2 \(\frac{x}{x+1}\)^4 - \frac{x^3}{(x+1)^4} - 32 \frac{x^2}{(x+1)^4} \right.\\
&\phantom{=\frac{\e^{-x}}{24}=} -\left. 72\frac{x}{(x+1)^4} - \frac{50}{(x+1)^4} + \frac{10}{x(x+1)^4} + \frac{2}{x^2(x+1)^4}\right],
\end{align*}
and hence, for $x\geq5$,
\begin{align*}
|Q'(x)|&\leq\frac{\e^{-x}}{24} \left[2\left(\frac56\right)^4 + \frac{5^3}{6^4} + \frac{32\cdot 5^2}{6^4} + \frac{72\cdot5}{6^4} + \frac{50}{6^4} + \frac{10}{5\cdot 6^4} + \frac{2}{5^2\cdot 6^4}\right] \leq \frac{\e^{-x}}{10}.
\end{align*}
Thus, the mean value theorem and the inequality \eqref{eq:aux7} give us
\begin{align*}
|q_{n+1}-q_n|\leq \sup_{x\in\left(W(n+1),~W(n+2)\right)}|Q'(x)|\(W(n+2)-W(n+1)\)\leq \frac{\e^{-W(n+1)}}{10(n+1)},
\end{align*}
which ends the proof.
\end{proof}

\section{Customizable bounds}

By the Cauchy theorem and the formula for the exponential  generating function \eqref{eq:egf} we get
	\begin{equation}\label{eq:2}
		B_n = \frac{n!}{2\pi i}\oint_{\Gamma} \frac{\egf{z}}{z^{n+1}}  d  z,
	\end{equation}
	where $\Gamma$ is the circle centred at the origin   $z = 0$ whose radius is some  $R>0$.
	Next, since the function ${\egf{z}}/{z^{n+1}}$ is holomorphic on  $\mathbb{C} \setminus \{0\}$, for any  $0<\eps <R$ we deform the circle into the sum of the following four curves
	\begin{align}
		\Gamma_1: \gamma_1(t) &= R + it, t \in (-\eps, \eps),\label{Gamma1}\\
		\Gamma_2: \gamma_2(t) &= R - t + i\eps, t \in (0, R - \sqrt{R^2 - \eps^2}),\label{Gamma2}\\
		\Gamma_3: \gamma_3(t) &= t - i\eps, t \in (\sqrt{R^2 - \eps^2}, R),\label{Gamma3}\\
		\Gamma_4: \gamma_4(t) &= R\e^{it}, t \in (\delta, 2\pi - \delta), \delta = \arcsin\left(\frac{\eps}{R}\right).\label{Gamma4}
	\end{align}

\begin{adjustbox}{center}
		\begin{tikzpicture}[remember picture, scale=0.7]
			\draw[ultra thick, black, ->] (-7,0) -- (9,0);
			\draw[ultra thick, black, ->] (0,-6) -- (0,6);
			\node[below left] at (9,0) {\Large$\Re$};
			\node[below right] at (0,6) {\Large$\Im$};
			\node[color=red, above right] at (3, 4) {\Large$\Gamma_4$};
			\node[color=blue, above right] at (5, -1.4) {\Large$\Gamma_1$};
			\node[color=red, above right] at (4.5, 2) {\Large$\Gamma_2$};
			\node[color=red, below right] at (4.5, -2.1) {\Large$\Gamma_3$};
			\draw[color=blue, ultra thick] (5, -2) -- (5, 2);
			\draw[color=red, ultra thick] (4.56, 2) -- (5, 2);
			\draw[color=red, ultra thick] (4.56, -2) -- (5, -2);
			\draw[color=black, thin] (5, 0) circle (2);
			\draw[color=black, thin] (0, 0) -- (3, 4);
			\draw[color=black, thin] (5, 0) -- (6, 1.73);
			\node[color=black, left] at (2, 3) {\Large$R$};
			\node[color=black, left] at (6.2, 0.8) {\Large$\eps$};
			\draw[color=black, thin, dashed] (0, 0) circle (5);

			\tikzstyle{reverseclip}=[insert path={(current page.north east) --
				(current page.south east) --
				(current page.south west) --
				(current page.north west) --
				(current page.north east)}
			]
			\begin{pgfinterruptboundingbox}
				\path [clip] (6, -2) rectangle(4, 2) [reverseclip];
			\end{pgfinterruptboundingbox}
			
			\draw[color=red, ultra thick] (0, 0) circle (5);
		\end{tikzpicture}
	\end{adjustbox}
	This gives us
	\begin{equation}\label{eq:3}
		B_n =  J_1 + J_2 + J_3 + J_4,
	\end{equation}
	where
	\begin{equation*}
		J_i =\frac{n!}{2\pi i} \int_{\Gamma_i} \frac{\egf{z}}{z^{n+1}}  d  z.
	\end{equation*}
	For any fixed positive integer $n$ the optimal $R$ turns out to be $W(n+1)$, hence, from now on, we denote 
	$$R=R(n)=W(n+1).$$

\begin{fact}\label{prop:J_1} For $R\geq5$ and  $\eps<1$ such that $\eps^2{\e^R}>5$, we have
\begin{align}
&\left|\frac{J_1}{E_n}-\(1-\frac{\e^{-R}}{12}\frac{1-\tfrac3{2R}-\tfrac{10}{R^2}-\tfrac{9}{R^3}+\tfrac1{R^4}}{\(1+\tfrac1R\)^{3}}\)\right|\nonumber\\
&\leq\sqrt{\frac 2\pi}\frac1{\eps}\exp\({{-\frac12\eps^2\e^R-\frac12R}}\)+ \frac{6}{5}\exp\(\frac{\e^R\eps^4}{22}-2R\)+\frac{\eps^{7}\exp\(-\tfrac12\eps^2{\e^R}+\tfrac52R\)}{30\(\eps^2{\e^R}-5\)}\label{prop1}.
\end{align}
\end{fact}

\begin{proof} Using \eqref{Gamma1}, we have
	\begin{align*}
		J_1 &= \frac{n!}{2\pi i}\int_{\Gamma_1} \frac{\egf{z}}{z^{n+1}}  d  z 
		= \frac{n!}{2\pi\e} \int_{-\eps}^{\eps} \e^{F(t)}  d  t,
	\end{align*}
	where 
	$$F(t) = \e^{R + it} - (n+1)\Log(R+it),\ \ \ \ t\in\R.$$  
	Here, $\Log(\cdot)$ stands for the principal value of the complex logarithm. It is easy to calculate that 
		\begin{equation}\label{ind:1}
			F^{(k)}(t) = i^k\left(\e^{R+it} + (-1)^k (n+1) \frac{(k-1)!}{(R+it)^k}\right),\ \ \ \ k\geq1.
		\end{equation}
	Taking $R$ such that $R\e^R = n+1$ (or equivalently $R = W(n+1)$), for $|t|<R$ we have
	\begin{align*}
		F(t)&=\sum_{k=0}^{\infty} \frac{F^{(k)}(0)}{k!}t^k 
		= \e^R - (n+1)\ln(R) + \sum_{k=1}^{\infty} i^k\left(\e^{R} + (-1)^k (n+1) \frac{(k-1)!}{R^k}\right) \frac{t^k}{k!}\\
		&= \e^R - (n+1)\ln(R) + \sum_{k=2}^{\infty} i^k\e^{R}\left(1 + (-1)^k  \frac{(k-1)!}{R^{k-1}}\right) \frac{t^k}{k!},
	\end{align*}
where the last equality is a consequence of $F'(0) = i\left(\e^R - \frac{R\e^R}{R}\right) = 0$. Next, we  decompose
	\begin{align}\label{eq:aux4}
	\e^{F(t)}
	&=\e^{F(0)+F''(0)t^2/2}+\e^{F(0)+F''(0)t^2/2}\(\e^{F(t)-F(0)-F''(0)t^2/2}-1\).
	\end{align}
	The two terms on the right-hand side  are crucial for the behaviour of the Bell numbers. The first one drives the asymptotic form. Indeed, when integrated over whole real line, we get
	\begin{equation}\label{eq:aux5}
	\frac{n!}{2\pi \e}\int_\R \e^{F(0)+F''(0)t^2/2}dt=\frac{n!\e^{F(0)}}{\sqrt{2\pi}\e\sqrt{|F''(0)|}}=\frac{n!\e^{\e^R-(n+1)\ln R-\tfrac12R}}{\sqrt{2\pi}\e\sqrt{1+\frac1R}}=E_n,
	\end{equation}
and consequently
\begin{align*}
&\left|\frac{n!}{2\pi\e}\int_{-\eps}^{\eps}\e^{F(0)+F''(0)t^2/2} d  t-E_n\right|\\
&=\left|\frac{n!}{\pi\e}\e^{F(0)}\int_{\eps}^\infty \e^{F''(0)t^2/2} d  t\right|=\frac{n!}{\sqrt{2\pi}\e}\frac{\e^{F(0)}}{\sqrt{|F''(0)|}}\erfc\(\eps\sqrt{|F''(0)|/2}\)\\
&\leq \frac{n!}{\pi\e\eps{|F''(0)|}}\e^{F(0)-\eps^2|F''(0)|/2}\leq \frac{n!}{\pi\e\eps\sqrt{\e^R|F''(0)|}}\e^{F(0)-\eps^2\e^R/2}\\
&= E_n\sqrt{\frac 2\pi}\frac{\e^{-\eps^2\e^R/2-R/2}}{\eps},
\end{align*}	
where we used \eqref{eq:aux5} and the inequalities $F''(0)\leq -\e^R$ and  $\erfc(t)  \leq {\e^{-t^2}}/{t\sqrt{\pi}}$, $t>0$.
The latter term in \eqref{eq:aux4} determines the convergence rate, and will be therefore bounded very precisely. Since $\e^{\overline{F(z)}} = \overline{\e^{F(z)}}$, we may write
\begin{align*}
&\frac{n!}{2\pi\e} \int_{-\eps}^{\eps} \e^{F(0)+F''(0)t^2/2}\(\e^{F(t)-F(0)-F''(0)t^2/2}-1\)  d  t\\
&=\frac{n!}{\pi\e} \int_{0}^{\eps} \Re\[\e^{F(0)+F''(0)t^2/2}\(\e^{F(t)-F(0)-F''(0)t^2/2}-1\) \] d  t\\
&=\frac{n!}{\pi\e} \int_{0}^{\eps} \e^{F(0)+F''(0)t^2/2}\(\cos(F_1(t))\e^{F_{2}(t)}-1\)  d  t,
\end{align*}
where 
\begin{align*} 
F_1(t)&=\Im\(F(t)-F(0)-F''(0)t^2/2\)= \sum_{k=1}^{\infty} (-1)^{k}\e^{R}\left(1 -   \frac{(2k)!}{R^{2k}}\right) \frac{t^{2k+1}}{(2k+1)!},\\
F_2(t)&=\Re\(F(t)-F(0)-F''(0)t^2/2\)= \sum_{k=2}^{\infty} (-1)^k\e^{R}\left(1 +   \frac{(2k-1)!}{R^{2k-1}}\right) \frac{t^{2k}}{(2k)!}.
\end{align*}
Let us now estimate the integrand in the last integral. This is a delicate task, since $\cos\left(F_1(t)\right)$ is decreasing, while $\e^{F_{2}(t)}$ is increasing. 
We will approximate  $\cos(F_1(t))\e^{F_{2}(t)}-1$ by $F^{(4)}(0)\tfrac{t^4}{4!}-(F^{(3)}(0))^2\tfrac{t^6}{2(3!)^2}$ and estimate the difference between them. This approach is a consequence of multivariate Taylor theorem. We have
\begin{align*}
&\(\cos(F_1(t))\e^{F_{2}(t)}-1\)-\(F^{(4)}(0)\tfrac{t^4}{4!}-(F^{(3)}(0))^2\tfrac{t^6}{2(3!)^2}\)\\
&=\(\cos(F_1(t))\e^{F_{2}(t)}-1\)-\(F_2(t)-\tfrac12(F_1(t))^2\)\\
&\ \ \ +\(F_2(t)-\tfrac12(F_1(t))^2\)-\(F^{(4)}(0)\tfrac{t^4}{4!}-(F^{(3)}(0))^2\tfrac{t^6}{2(3!)^2}\)\\
&=\e^{F_{2}(t)}\(\cos(F_1(t))-1+\tfrac12(F_1(t))^2\)+\(\e^{F_2(t)}-1-F_2(t)\)+\tfrac12(F_1(t))^2\(1-\e^{F_2(t)}\)\\
&\ \ \ +\(F_2(t)-F^{(4)}(0)\tfrac{t^4}{4!}\)-\frac12\(F_1(t)-F^{(3)}(0)\tfrac{t^3}{3!}\)\(F_1(t)+F^{(3)}(0)\tfrac{t^3}{3!}\).
\end{align*}
One can easily verify that absolute values of coefficients in both of the series defining $F_1$ and $F_2$ are decreasing for $R\geq5$ and $|t|<1$, hence we may bound them by their first terms:
\begin{align*}
|F_1(t)|&\leq \frac{|F^{(3)}(0)|t^3}{3!}\leq\frac{\e^R|t|^3}{3!}\(1-\frac{2}{R^2}\)\leq \frac{\e^R|t|^3}{6},\\
|F_2(t)|&\leq \frac{|F^{(4)}(0)|t^4}{4!}\leq \frac{\e^Rt^4}{4!}\(1+\frac{3!}{5^3}\)\leq \frac2{45}{\e^Rt^4}.
\end{align*}
Analogously we also get
\begin{align*}
|F_1(t)-F^{(3)}(0)\tfrac{t^3}{3!}|&\leq \frac{|F^{(5)}(0)|t^5}{5!}\leq \frac{\e^R|t|^5}{5!},\\
|F_2(t)-F^{(4)}(0)\tfrac{t^4}{4!}|&\leq \frac{|F^{(6)}(0)|t^6}{6!}\leq \frac{\e^Rt^6}{6!}\(1+\frac{5!}{5^5}\)\leq \frac{27}{26}\frac{\e^Rt^6}{6!}.
\end{align*}
Using also inequalities  $|\cos x-1+\tfrac12x^2|<\tfrac1{4!}x^4$, $|\e^x-1|\leq x\e^x$ and $|\e^x-1-x|\leq \frac12x^2\e^x$, $x>0$, we obtain 
\begin{align*}
&\left|\(\cos(F_1(t))\e^{F_{2}(t)}-1\)-\(F^{(4)}(0)\tfrac{t^4}{4!}-(F^{(3)}(0))^2\tfrac{t^6}{2(3!)^2}\)\right|\\
&\leq \e^{F_2(t)}\(\frac{(F_1(t))^4}{4!}+\frac{(F_2(t))^2}{2}+\frac12(F_1(t))^2F_2(t)\)+\frac{|F^{(6)}(0)|t^6}{6!}+\frac12\frac{|F^{(5)}(0)|t^5}{5!}2\frac{|F^{(3)}(0)|t^3}{3!}\\
&\leq \e^{2\e^R\eps^4/45}\(\frac{\e^{4R}t^{12}}{4!6^4}+\frac{2\e^{2R}t^8}{(45)^2}+\frac{\e^{2R}t^{6}}{6^2}\frac{\e^Rt^4}{45}\)+\frac{27}{26}\frac{\e^Rt^6}{6!}+\frac{\e^{2R}t^8}{6!}\\
&\leq \e^{\e^R\eps^4/22}\(\frac{\e^{4R}t^{12}}{4!6^4}+\frac{\e^{2R}t^8}{420}+\frac{\e^{2R}t^{6}}{6^2}\frac{\e^Rt^4}{45}+\frac{27}{26}\frac{\e^Rt^6}{6!}\),
\end{align*}
where we used $2\cdot45^{-2}+(6!)^{-1}\leq 1/416$. Consequently
\begin{align*}
&\frac{n!}{\pi\e} \int_{0}^{\eps} \e^{F(0)+F''(0)t^2/2}\left|\(\cos(F_1(t))\e^{F_{2}(t)}-1\)-\(F^{(4)}(0)\tfrac{t^4}{4!}-(F^{(3)}(0))^2\tfrac{t^6}{2(3!)^2}\)\right|  d  t\\
&\leq \frac{n!\e^{F(0)}}{\pi\e}\e^{2\e^R\eps^4/45} \int_{0}^{\infty} \e^{F''(0)t^2/2}\(\frac{\e^{4R}}{4\cdot6^5}t^{12}+\frac{\e^{2R}}{416}t^8+\frac{\e^{3R}}{1620}t^{10}+\frac{3\e^{R}}{2080}t^6\)  d  t.
\end{align*}
Applying the bound
$$\int_{0}^{\infty}\e^{F''(0)t^2/2}t^k d  t=\frac12\left|\frac{2}{F''(0)}\right|^{(k+1)/2}\Gamma\(\frac{k+1}{2}\)\leq \frac{2^{(k-1)/2}\e^{-Rk/2}}{\sqrt{|F''(0)|}}\Gamma\(\frac{k+1}{2}\),$$
and verifying 
$$\Gamma\(\frac72\)=\frac{15\sqrt\pi}{8}, \ \ \  \Gamma\(\frac92\)=\frac{7!!\sqrt\pi}{16}. \ \ \ \Gamma\(\frac{11}2\)=\frac{9!!\sqrt\pi}{32}, \ \ \  \Gamma\(\frac{13}2\)=\frac{11!!\sqrt\pi}{64},$$
we estimate further 
\begin{align*}
&\leq \frac{n!\e^{F(0)}}{{\pi}\e\sqrt{|F''(0)|}}\e^{\e^R\eps^4/22}{\e^{-2R}}\(\frac{2^{11/2}\Gamma\(\frac{13}{2}\)}{4\cdot6^5}+\frac{2^{7/2}\Gamma\(\frac{9}{2}\)}{416}+\frac{2^{9/2}\Gamma\(\frac{11}{2}\)}{1620}+\frac{3\cdot2^{5/2}\Gamma\(\frac{7}{2}\)}{2080}\) \\
&=E_n\sqrt{\frac2{\pi}}\e^{\e^R\eps^4/22}{\e^{-2R}}\(\frac{\Gamma\(\frac{13}{2}\)}{2\cdot3^5\sqrt2}+\frac{\Gamma\(\frac{9}{2}\)}{26\sqrt2}+\frac{8\Gamma\(\frac{11}{2}\)}{405\sqrt2}+\frac{3\Gamma\(\frac{7}{2}\)}{260\sqrt2}\) \\
&=E_n\e^{\e^R\eps^4/22}{\e^{-2R}}\(\frac{385}{1152}+\frac{105}{416}+\frac{21}{36}+\frac9{416}\) \\
&\leq \frac{6}5E_n\e^{\e^R\eps^4/22}{\e^{-2R}}.
\end{align*}
Eventually, substituting $u=F''(0)t^2/2$, we get
\begin{align*}
&\frac{n!}{\pi\e} \int_{0}^{\eps} \e^{F(0)+F''(0)t^2/2}\(F^{(4)}(0)\tfrac{t^4}{4!}-(F^{(3)}(0))^2\tfrac{t^6}{2(3!)^2}\)  d  t\\
&=\frac{n!\e^{F(0)}}{\pi\e} \int_{0}^{\eps^2{|F''(0)|/2}} \e^{-u}\(\frac{F^{(4)}(0)}{48}\left|\frac{2}{F''(0)}\right|^{5/2}u^{3/2}-\frac{(F^{(3)}(0))^2}{144}\left|\frac{2}{F''(0)}\right|^{7/2}u^{5/2}\)  d  t\\
&=\frac{n!\e^{F(0)}}{\pi\e} \(F^{(4)}(0)\frac{\Gamma\(5/2\)\sqrt2}{12|F''(0)|^{5/2}}-(F^{(3)}(0))^2\frac{\Gamma\({7/2}\)\sqrt2}{18|F''(0)|^{7/2}}\) \\
&\ \ \ \  -\frac{n!\e^{F(0)}}{\pi\e} \int_{\eps^2{|F''(0)|/2}}^{\infty} \e^{-u}\(F^{(4)}(0)\frac{u^{3/2}\sqrt2}{12|F''(0)|^{5/2}}-(F^{(3)}(0))^2\frac{u^{5/2}\sqrt2}{18|F''(0)|^{7/2}}\)  d  t\\
&=E_n \(\frac{F^{(4)}(0)}{8|F''(0)|^{2}}-\frac{5(F^{(3)}(0))^2}{24|F''(0)|^{3}}\) \\
&\ \ \ \ -\frac{E_n}{\sqrt\pi} \(\frac{F^{(4)}(0)}{6|F''(0)|^{2}}\Gamma\(\frac52, \eps^2{|F''(0)|/2}\)-\frac{(F^{(3)}(0))^2}{9|F''(0)|^{3}}\Gamma\(\frac72, \eps^2{|F''(0)|/2}\)\),
\end{align*}
where $\Gamma(p,x)$ stands for the incomplete Gamma function. From Theorem 2.1 in \cite{BC} we have
$$\Gamma(p,x)\leq \frac{x^p\e^{-x}}{x-p+1},\ \ \ \ p\geq1,x>p-1.$$
 Using this and the equalities $|F^{(k)}(0)|\leq \e^R\(1+\tfrac1R\)$, $R>\sqrt6$, for $k=3$ and $k=4$, we get for $\eps^2{\e^R}>5$ 
\begin{align*}
&\left|\frac{E_n}{\sqrt\pi} \(\frac{F^{(4)}(0)}{6|F''(0)|^{2}}\Gamma\(\frac52, \eps^2{|F''(0)|/2}\)-\frac{(F^{(3)}(0))^2}{9|F''(0)|^{3}}\Gamma\(\frac72, \eps^2{|F''(0)|/2}\)\)\right|\\
&\leq{E_n} \frac{\exp\(-\tfrac12\eps^2{|F''(0)|}\)}{\sqrt\pi\(\eps^2{|F''(0)|/2}-\tfrac52\)}\(1+\tfrac1R\)^{5/2} \(\frac{1}{6}\frac{\eps^{5}\e^{3R/2}}{2^{5/2}}+\frac{1}{9}\frac{\eps^{7}\e^{5R/2}}{2^{7/2}}\)\\
&={E_n} \frac{\exp\(-\tfrac12\eps^2{|F''(0)|}\)}{\eps^2{|F''(0)|}-5}\eps^{7}\e^{5R/2}\frac{(1+\tfrac1R)^{5/2}}{\sqrt\pi}\frac{2}{9\cdot 2^{7/2}}\(\frac{3}{\eps^2\e^R}+1\)\\
&\leq{E_n} \frac{\eps^{7}\exp\(-\tfrac12\eps^2{\e^R}+\tfrac52R\)}{\(\eps^2{\e^R}-5\)}\frac{(1+\tfrac15)^{5/2}}{\sqrt\pi}\frac{1}{9\cdot 2^{5/2}}\(\frac{3}{5}+1\)\\
&\leq{E_n} \frac{\eps^{7}\exp\(-\tfrac12\eps^2{\e^R}+\tfrac52R\)}{30\(\eps^2{\e^R}-5\)}.
\end{align*}
Additionally, 
\begin{align*}
\frac{F^{(4)}(0)}{8|F''(0)|^{2}}-\frac{5(F^{(3)}(0))^2}{24|F''(0)|^{3}}&=\e^{-R}\frac{\(1+\tfrac{6}{R^3}\)}{8\(1+\tfrac1R\)^{2}}-\e^{-R}\frac{5\(1-\tfrac2{R^2}\)^2}{24\(1+\tfrac1R\)^{3}}\\
&=\e^{-R}\(\frac{3\(1+\tfrac{6}{R^3}\)\(1+\tfrac1R\)}{24\(1+\tfrac1R\)^{3}}-\frac{5\(1-\tfrac2{R^2}\)^2}{24\(1+\tfrac1R\)^{3}}\)\\
&=-\frac{\e^{-R}}{12}\frac{1-\tfrac3{2R}-\tfrac{10}{R^2}-\tfrac{9}{R^3}+\tfrac1{R^4}}{\(1+\tfrac1R\)^{3}}.
\end{align*}
\end{proof}

\begin{fact}\label{Fakt2}
	For $0<\eps<\tfrac12$ and $R\geq 5$ it holds
	\begin{equation*}
		\left|\frac{J_2 + J_3}{E_n}\right| \leq {\exp\({\e^R(\cos\eps-1)-{\tfrac12R(1-\eps^2)}}\)}.
	\end{equation*}
\end{fact}

\begin{proof}
	First, let us write $J_2+J_3$ in the following form
	$$J_2+J_3= \frac1{2\pi i}\int_{\sqrt{R^2-\eps^2}}^{R} \frac{\egf{t-i\eps}}{(t-i\eps)^{n+1}} - \frac{\egf{t+i\eps}}{(t+i\eps)^{n+1}}  d  t.$$
	Next, denoting  $\alpha_t=\arg(t+i\eps)$, we get
	\begin{align*}
		\frac{\egf{t-i\eps}}{(t-i\eps)^{n+1}} - \frac{\egf{t+i\eps}}{(t+i\eps)^{n+1}} 
		&= \frac{\exp\({\e^t\cos\eps}\)}{\e(t^2+\eps^2)^{(n+1)/2}} \(\e^{i(-\sin\eps+\alpha_t)}- \e^{i(\sin\eps-\alpha_t)}\)\\\
		&= \frac{\exp\({\e^t\cos\eps}\)}{\e(t^2+\eps^2)^{(n+1)/2}} 2i\sin\(\alpha_t-\sin\eps\).
	\end{align*}
	This gives us
	\begin{align*}
		\left|J_2+J_3\right|
		&= \left|\frac1{2\pi i}\int_{\sqrt{R^2-\eps^2}}^{R} \frac{\egf{t-i\eps}}{(t-i\eps)^{n+1}} - \frac{\egf{t+i\eps}}{(t+i\eps)^{n+1}}  d  t\right|\\
		&\leq \frac{n!}{\pi\e}\int_{\sqrt{R^2-\eps^2}}^{R} \frac{\exp(\e^t\cos(\eps))}{\e(\sqrt{t^2 + \eps^2})^{n+1}}  d  t\\
		&\leq \frac{n!}{\pi \e R^{n+1}} \int_{\sqrt{R^2-\eps^2}}^{R} \frac{\e^t}{\e^{\sqrt{R^2-\eps^2}}}{\exp(\e^t\cos(\eps))}  d  t\\
		&= \frac{n!}{\pi\e R^{n+1}\e^{\sqrt{R^2-\eps^2}}\cos(\eps)} \(\e^{\e^R\cos\eps}-\e^{\e^{\sqrt{R^2-\eps^2}}\cos\eps}\)\\
		&\leq \frac{n!\exp\({\e^R\cos\eps-{R(1-\tfrac12\eps^2)}}\)}{\pi\e R^{n+1}\cos\eps},
	\end{align*}
	where in the last inequality we simply omitted the latter term in the parentheses. 
	Consequently, we may write
	\begin{align*}
	\left|\frac{J_2+J_3}{E_n}\right|&\leq \frac{n!\exp\({\e^R\cos\eps-{R(1-\tfrac12\eps^2)}}\)}{\pi\e R^{n+1}\cos\eps}\frac{R^{n} \sqrt{2\pi(n+1) (R + 1)}}{n!\egf{R}}\\
	&\leq {\exp\({\e^R(\cos\eps-1)-{\tfrac12R(1-\eps^2)}}\)}\frac{ \sqrt{2\pi  ( 1+ \tfrac1R)}}{\pi \cos\eps}.
	\end{align*}  
	Finally, the bound
	\begin{align*}
	\frac{ \sqrt{2\pi  ( 1+ \tfrac1R)}}{\pi \cos\eps}\leq \frac{ \sqrt{2\pi  ( 1+ \tfrac14)}}{\pi \cos(1/2)}\leq1
	\end{align*}
	ends the proof.
\end{proof}

\begin{fact}\label{Fakt1} For $R\geq 4$ and $\varepsilon <1/2$ we have
	\begin{align*}
		\left|\frac{J_4}{E_n}\right| &\leq 3\[\exp\(\e^{R}(\cos(\eps)-1)-\tfrac12R\)+R\exp\(-\tfrac2R\e^{R}+\tfrac12R\)\] .
	\end{align*}
\end{fact}
\begin{proof}
Let us rewrite 
\begin{align*}
\left|J_4\right|&=\left|\frac{n!}{2\pi i}\int_{\delta}^{2\pi - \delta} \frac{\egf{R\e^{it}}}{\left(R\e^{it}\right)^{n+1}} Ri\e^{it}  d  t\right|=\frac{n!}{2\pi\e R^n}\left|\int_{\delta}^{2\pi - \delta} {\exp\(\e^{R\e^{it}}\)} \e^{-int}  d  t\right|\\
&\leq\frac{n!}{2\pi \e R^n}\int_{\delta}^{2\pi - \delta}\left|{\exp\(\e^{R\e^{it}}\)} \e^{-int}\right| d  t\\
&=\frac{n!}{\pi \e R^n}\int_{\delta}^{\pi }{\exp\(\e^{R\cos t}\cos(R\sin t)\)} d  t,
\end{align*}
where $\delta = \arcsin\left(\frac{\eps}{R}\right)$. Now we split the last integral into three ones over the intervals $(\delta,\tfrac6{5R})$, $(\tfrac6{5R},\tfrac4R)$ and $(\tfrac4R,\pi)$. 
	\begin{align*}
J_{4,1}&:=\left|\frac{n!}{\pi \e R^n}\int_{\delta}^{6/5R} \exp\(\e^{R\cos t}\cos(R\sin t)\)  d  t\right|\\
	&\leq\frac{n!}{\pi \e R^n}\int_{\delta}^{6/5R}\frac{\e^{R}\sin(R\sin t)R\cos t}{\e^{R}\sin\(R\sin\delta\)R\cos \tfrac6{5R}} \exp\(\e^{R}\cos(R\sin t)\)  d  t\\
		&\leq\frac{n!}{\pi \e R^{n+1}\e^{R}\sin\(\eps\)\cos \tfrac6{5R}} \[-\exp\(\e^{R}\cos(R\sin t)\)\]_{t=\delta}^{t=6/5R}\\
			&\leq\frac{n!}{\pi \e R^{n+1}\(\eps(1-\tfrac{\eps^2}6)\)\cos \tfrac6{5R}}\exp\(\e^{R}\cos(\eps)-R\)\\
				&\leq\frac52\frac{n!}{\pi \e R^{n+1}}\exp\(\e^{R}\cos(\eps)-R\).
\end{align*}
where we used $\sin x\geq x(1-x^2/6)$, $0\leq x\leq 1$. Next, 
	\begin{align*}
J_{4,2}&:=\frac{n!}{\pi \e R^n}\int_{6/5R}^{4/R } \exp\(\e^{R\cos (t)}\cos(R\sin(t))\)  d  t\\
&\leq \frac{3n!}{\pi \e R^{n+1}} \exp\(\e^{R}\cos(R\sin(6/5R))\) \\
&\leq \frac{3n!}{\pi \e R^{n+1}} \exp\(\e^{R}\cos\(\tfrac65-\tfrac6{25R^2}\)\) \\
&\leq \frac{3n!}{\pi \e R^{n+1}} \exp\(\tfrac12\e^{R}\),
\end{align*}
where in the last inequality we used 
$$\frac65-\frac6{25R^2}\in\(\frac65-\frac{1}{400},\frac65\)\subseteq\(\frac{\pi}3,\frac65\) .$$
Thus, using the inequality $\cos x\geq 7/8$, $|x|\leq 1/2$, we get 
\begin{align*}
J_{4,1}+J_{4,2}&\leq \frac52\frac{n!}{\pi \e R^{n+1}}\exp\(\e^{R}\cos(\eps)-R\)\(1+\frac65 \exp\(\e^{R}(\tfrac12-\cos(\varepsilon))+R\)\)\\
&\leq \frac52\frac{n!}{\pi \e R^{n+1}}\exp\(\e^{R}\cos(\eps)-R\)\(1+\frac65 \exp\(-\tfrac38\e^{R}+R\)\)\\
&\leq \frac{n!}{ \e R^{n+1}}\exp\(\e^{R}\cos(\eps)-R\).
\end{align*}

Eventually, by the inequalities $\cos(x) \leq 1 - \frac{1}{4}x^2$ and $1-\e^{-x}\geq \frac12x$, $0\leq x\leq1$, we obtain
\begin{align*}
J_{4,3}&:=\frac{n!}{\pi \e R^n}\int_{4/R}^{\pi }{\exp\(\e^{R\cos t}\cos(R\sin t)\)} d  t\leq \frac{n!}{\pi \e R^n}\int_{4/R}^{\pi } \exp\(\e^{R\cos (4/R)}\)  d  t\\
&\leq   \frac{n!}{\e R^n}\exp\(\e^{R-4/R}\)=\frac{n!}{\e R^n}\exp\(\e^R-\e^R(1-\e^{-4/R})\)\leq \frac{n!}{\e R^n}\exp\(\e^R(1-\tfrac2R)\).
\end{align*}
Summing up, we arrive at
\begin{align*}
\left|\frac{J_4}{E_n}\right| &\leq \(J_{4,1}+J_{4,2}+J_{4,3}\)\frac{R^{n} \sqrt{2\pi(n+1) (R + 1)}}{n!\egf{R}}\\
&\leq \(\exp\(e^{R}(\cos(\eps)-1)-R\)+R\exp\(-\tfrac2R\e^{R}\)\)\frac1R \sqrt{2\pi R\e^R (R + 1)}\\
&\leq 3\[\exp\(\e^{R}(\cos(\eps)-1)-\tfrac12R\)+R\exp\(-\tfrac2R\e^{R}+\tfrac12R\)\] ,
\end{align*}
where be bounded $\sqrt{2\pi(1+\tfrac1R)}\leq \sqrt{2\pi(1+\tfrac14)}\leq3$. This ends the proof.
\end{proof}
	\begin{corollary}\label{cor:JJJ}
	 For $R\geq 5$ and $\varepsilon <1/2$ we have
	\begin{align*}
		\left|\frac{J_2+J_3+J_4}{E_n}\right| &\leq 4\exp\(-\tfrac{11}{24}\eps^2\e^{R}-\tfrac38R\)+3R\exp\(-\tfrac2R\e^{R}+\tfrac12R\) .
	\end{align*}
	\end{corollary}
\begin{proof}
The assertion follows from Propositions \ref{Fakt2} and \ref{Fakt1} and the inequality 
$$1-\cos x\geq \frac12x^2-\frac1{4!}\leq \frac12x^2\(1-\frac1{12}x^2\)\leq \frac{11}{24}x^2,$$
valid for $|x|\leq1$.
\end{proof}

\section{Explicit bounds of the  Bell numbers}
We start this section  with some kind of master theorem, as it implies all the subsequent result. It establishes  a second order asymptotic of the Bell numbers together with the convergence rate. This provides us with the optimal order of the error in estimates,  allows to reach good constants in the first order bounds, and is crucial in approximating the ratio of the consecutive  Bell numbers.
\begin{theorem}\label{thm:master} For $n \geq 1$ it holds
\begin{equation*}
	\left|\frac{B_n}{E_n}-\(1-\frac{\e^{-R}}{12}\frac{1-\tfrac{3}{2R}-\tfrac{10}{R^2}-\tfrac{9}{R^3}+\tfrac1{R^4}}{\(1+\tfrac1R\)^{3}}\)\right|\leq 1.6\e^{-2R},
\end{equation*}
where $R=W(n+1)$.
\end{theorem}
\begin{remark}The proof of Theorem \ref{thm:master} relies on Proposition \ref{prop:J_1}, Corollary \ref{cor:JJJ} and proper choice of $\varepsilon$. In view of the middle term in \eqref{prop1}, the bound of order $\e^{-2R}$ is the best possible to obtain in this approach. Additionally, one requires $\varepsilon=\varepsilon(R)\leq C\e^{-R/4}$, for some $C>0$, in order to achieve this order.  Numerical analysis shows that, fixing $R=5$, lies around $1.4$. For simplicity, we choose $C=3/2$, as is does not impact much the final result.
\end{remark}
\begin{proof}[Proof of Theorem \ref{thm:master}]
We bound from above the left-hand side of the assertion by
\begin{align*}
	&\left|\frac{J_1}{E_n}-\(1-\frac{\e^{-R}}{12}\frac{1-\tfrac{3}{2R}-\tfrac{10}{R^2}-\tfrac{9}{R^3}+\tfrac1{R^4}}{\(1+\tfrac1R\)^{3}}\)\right|+\left|\frac{J_2+J_3+J_4}{E_n}\right|.
\end{align*}
Taking $\eps=\frac32 \e^{-R/4}$ we have $\eps^2{\e^R}>5$, and Proposition \ref{prop:J_1} gives us
\begin{align*}
	&\left|\frac{J_1}{E_n}-\(1-\frac{\e^{-R}}{12}\frac{1-\tfrac{3}{2R}-\tfrac{10}{R^2}-\tfrac{9}{R^3}+\tfrac1{R^4}}{\(1+\tfrac1R\)^{3}}\)\right|\\
	&\leq \e^{-2R} \left[\frac{2\sqrt{2}}{3\sqrt{\pi}} \exp\left(-\frac98 \e^\frac{R}{2} + \frac74 R\right) + \frac{6}{5}\e^{\frac{81}{352}}  + \frac{729}{1280} \frac{\exp\left(-\frac98 \e^{\frac{R}{2}} + \frac{11}4R\right)}{\frac94\e^{\frac{R}{2}} - 5}\right].
\end{align*}
The function of $R$ in the brackets is decreasing, we therefore bound it from above by the value at $R=5$ and get
\begin{equation*}
	\left|\frac{J_1}{E_n}-\(1-\frac{\e^{-R}}{12}\frac{1-\tfrac{3}{2R}-\tfrac{10}{R^2}-\tfrac{9}{R^3}+\tfrac1{R^4}}{\(1+\tfrac1R\)^{3}}\)\right|\leq 1.55\e^{-2R}.
\end{equation*}
Next, we employ Corollary \ref{cor:JJJ}. Since $\varepsilon<1/2$ for $R\geq5$, we have
\begin{align*}
		\left|\frac{J_2+J_3+J_4}{E_n}\right| &\leq \[4\exp\(-\tfrac{33}{32}\e^{R/2}+\tfrac{13}8R\)+3R\exp\(-\tfrac2R\e^{R}+\tfrac52R\)\]\e^{-2R} .
	\end{align*}
	Similarly as previously, it is easy to show that the expression in the  brackets is decreasing, and therefore
	\begin{align*}
		\left|\frac{J_4}{E_n}\right|
	 &\leq  \[4\exp\(-\tfrac{33}{32}\e^{5/2}+\tfrac{65}8\)+15\exp\(-\tfrac25\e^{5}+\tfrac{25}2\)\]e^{-2R}\\
	 &\leq 0.05\e^{-2R}.
	\end{align*}
	Combining the results we get the required bound for $R \geq 5$, which corresponds to $n \geq \lceil5\e^{5} - 1\rceil = 742$. Numerical calculations confirm the bound holds for $n \in\{ 1, \ldots, 741\}$ as well.
\end{proof}
As a first consequence of Theorem \ref{thm:master} we derive some  weaker, but more elegant and  convenient in applications, bounds on the ratio $B_n/E_n$.
\begin{prop}\label{prop:main} For $n\geq11$ we have 
	\begin{align*}
		\left|\frac{B_n}{E_n}-1\right|&\leq \frac{\e^{-W(n+1)}}{11}\leq \frac1{11}\frac{\ln n}{n}.
	\end{align*}
	Furthermore, for $n\geq311$ we have $B_n\leq E_n.$
\end{prop}
\begin{proof} Similarly as in \eqref{aux1}, from Theorem \ref{thm:master} we get for $R\geq5$
	\begin{align*}
		\left|\frac{B_n}{E_n}-1\right|&\leq \e^{-R}\left|- \frac{1-\tfrac{3}{2R}-\tfrac{10}{R^2}-\tfrac{9}{R^3}+\tfrac1{R^4}}{12\(1+\tfrac1R\)^{3}}\right|+\frac32\e^{-2R}\\
		&\leq \e^{-R}\(\frac{1}{12} + \frac32 \e^{-5}\)\leq \frac{\e^{-R}}{11}\leq \frac1{11}\frac{\ln (n+1)}{n+1}\leq \frac1{11}\frac{\ln n}{n},
	\end{align*}
	where we used $\tfrac{3}{2\cdot 5}+\tfrac{10}{5^2}+\tfrac{9}{5^3}<1$. Since $R \geq5 $ corresponds to $n \geq  742$, we confirm the proposition numerically for $n = 11, \ldots, 741$.
	
	The bound $B_n\leq E_n$ for $n\geq 742$ is a consequence of  Theorem \ref{thm:master} and the inequality \eqref{eq:q1}. Numerical calculations verify the bound for $n = 311, \ldots 742$.
\end{proof}
So far, the results are formulated by means of the asymptotic form $E_n$. However, our preferred form is $E^*$, since it allows to establish bounds for greater range of the indices, and is more convenient in some calculations. Below, we present relations between these two quantities.

\begin{lemma}\label{lem:EE}For $n\geq 1$ it holds
	$$ \left(1 - \frac1{2n}\right)E_n^*\leq E_n\leq E_n^*.$$
\end{lemma}
\begin{proof}
	Rewriting \eqref{eq:E_n} we obtain
	\begin{align}\nonumber
		E_n&=\frac{n!\exp\({\e^R-(n+1)\ln R}\)}{\e\sqrt{2\pi \e^R\(1+\frac1R\)}}=\frac{n!\exp\(\e^R-(n+1)\(\ln(n+1)-R\)\)}{\e\sqrt{2\pi\frac{n+1}{R}\(1+\frac1R\)}}\\\label{eq:aux2}
		&=\frac{n!\exp\({\int_{0}^{n+1}W(x)dx-(n+1)\[\ln(n+1)-1\]}\)}{\sqrt{2\pi\frac{n+1}{R}\(1+\frac1R\)}}\\\nonumber
		&=\frac{(n+1)!\exp\({\int_{0}^{n+1}W(x)dx}-W(n+1)\)}{\sqrt{2\pi(n+1)}\(\frac {n+1}  \e\)^{n+1}\sqrt{1+R}}\\\nonumber
		&=E_n^*\frac{(n+1)!}{\sqrt{2\pi(n+1)}\(\frac {n+1}  \e\)^{n+1}}\exp\({\int_{n}^{n+1}W(x)dx}-W(n+1)\),
	\end{align}
	Since $W$ is increasing and concave, and the inequality  $W(1)\geq1/2$ due to \eqref{eq:omega},
	the mean value theorem lets  us  bound
	\begin{align*}\int_{n}^{n+1}W(x)dx-W(n+1)&\leq - \int_n^{n+1}(n+1-x)dx\min_{z\in [n,n+1]}W'(z)\leq -\frac12W'(n+1)\\
	&=-\frac12\frac{1}{n+1+\(1+\frac1{W(n+1)}\)}\leq -\frac1{4}\frac1{n+1},
	\end{align*}
	and similarly
	$${\int_{n}^{n+1}W(x)dx}-W(n+1)\geq -\frac12W'(n)\geq -\frac1{2n}.$$
	Using also the well known approximation 
	$$\e^{1/(12n+1)}\leq \frac{n!}{\sqrt{2\pi n}\(\frac {n}  \e\)^{n}}\leq \e^{1/12n},$$
	and the inequality $e^x\geq 1+x$, $x\in\R$, we get
	\begin{align*}
		\frac{E_n}{E_n^*}&\leq \e^{\frac1{n+1}\(\frac1{12}-\frac14\)}\leq 1,\\
		\frac{E_n}{E_n^*}&\geq \e^{-1/(2n)}\geq 1-\frac{1}{2n},
	\end{align*}
	which ends the proof.
\end{proof}

From Theorem \ref{thm:master} we can directly deduce that
\begin{equation}\label{aux1}
	\e^R \left(\frac{B_n}{E_n} - 1 \right) =  - \frac{1-\tfrac{3}{2R}-\tfrac{10}{R^2}-\tfrac{9}{R^3}+\tfrac1{R^4}}{12\(1+\tfrac1R\)^{3}} + O(\e^{-R}).
\end{equation}
Recalling that $\e^R=(n+1)/W(n+1)$ and $W(n+1)\sim \ln n$, by Lemma \ref{lem:EE} we get the asymptotics of the relative error of estimating $B_n$ by $E_n$ or $E_n^*$.
\begin{corollary}\label{cor:errorlim} We have 
$$\lim_{n\rightarrow\infty}\frac{ n}{\ln n}\left(\frac{B_n}{E_n}-1\right)=\lim_{n\rightarrow\infty}\frac{ n}{\ln n}\left(\frac{B_n}{E_n^*}-1\right)=-\frac{1}{12}.$$ 
\end{corollary}
This shows that the rate of convergence obtained in Proposition \ref{prop:main} and in the next Theorem  is of the optimal order. 
\begin{theorem}\label{prop:main2} For $n\geq2$ it holds 
	\begin{align*}
		\(1-\frac15\frac{\ln n}{n}\)E_n^*\leq B_n\leq E_n^*.
	\end{align*}
	
\end{theorem}
\begin{proof}
	From Lemma \ref{lem:EE} and previous proposition we conclude for $n\geq311$
	\begin{align*}
		B_n\leq E_n\leq E_n^*.
	\end{align*}
	Next, by Lemma \ref{lem:EE} and Proposition \ref{prop:main} we get  for $n \geq 311$
	\begin{align*}
		\frac{B_n}{E_n^*}-1&=\(\frac{B_n}{E_n}-1\)\frac{E_n}{E_n^*}+\frac{E_n}{E_n^*}-1\geq-\left|\frac{B_n}{E_n}-1\right|-\left|\frac{E_n}{E_n^*}-1\right|\\
		&\geq-\(\frac1{11}\frac{\ln n}{n}+\frac1{2n}\)=-\frac{\ln n}{n}\(\frac1{11}+\frac{1}{2\ln n}\)\\
		&\geq -\frac{\ln n}{n}\(\frac1{11}+\frac{1}{2\ln (311)}\)\geq -\frac1{5}\frac{\ln n}{n}.
	\end{align*}
	For $n = 2, \ldots, 310$ we verify the assertion numerically.
\end{proof}

At the end of this section we provide bounds by means of elementary functions. They do not match the asymptotic forms, however, they give better comprehension of the magnitude of the Bell numbers and seem sufficient for most of the applications. 
\begin{prop}\label{prop:ele} We have
	\begin{align}\label{eq:expl1}
	\(\frac1\e\frac{n}{\ln n}\)^n\leq B_n&
	\leq \(\frac34\frac{n}{\ln n}\)^n,&&n\geq2,\\[5pt]\label{eq:expl2}
	B_n&	\leq \(\frac1\e\frac{n}{\ln n}\big(1+\tfrac{3\ln\ln n}{\ln n}\big)\)^n,&&n\geq6.
	\end{align}

%
\end{prop}
\begin{proof} Since  $\e^{W(n)}=n/W(n)$, we write
	\begin{align}\nonumber
		E_n^*&=\frac{\exp\({\frac n{W(n)}+n(\ln n-\ln W(n))-(n+1)}\)}{ \sqrt{1+W(n)} }\\\label{eq:aux14}
		&=\(\frac{n}{\e\ln n}\)^n{\exp\[n\({\frac1{W(n)}+\ln\(\tfrac{W(n)+\ln W(n)}{ W(n)}\)-\frac1n-\frac1{2n}\ln\big(1+W(n)\big)}\)\]}\\\label{eq:aux15}
			&\leq\[\frac{n}{\e\ln n}\e^{1/W(n)}\(1+\tfrac{\ln W(n)}{ W(n)}\)\]^n.
	\end{align}
	In order to get the upper bound in \eqref{eq:expl1} we estimate for $W(n)\geq \e$ (equivalently $n\geq \lceil \e\e^{\e}\rceil =42$)
	\begin{align*}
		\frac1{\e}\e^{1/{W(n)}}\(1+\tfrac{\ln W(n)}{ W(n)}\)\leq \frac{\e^{1/\e}}\e\(1+\frac1{\e}\)\leq\frac34,
	\end{align*}
	as required. Furthermore, if $W(n)\geq5$ (equivalently $n\geq 743$ ), by \eqref{eq:eW=x/W} and \eqref{eq:Wbounds0}  we have
	\begin{align*}
	\e^{1/{W(n)}}\(1+\tfrac{\ln W(n)}{ W(n)}\)&=1+\(\e^{1/{W(n)}}-1\)+\e^{1/{W(n)}}\frac{\ln W(n)}{ W(n)}\\
	&\leq 1+\e^{1/5}\(\frac1{W(n)}+\frac{\ln W(n)}{ W(n)}\)\\
	&= 1+\e^{1/5}\frac{\ln W(n)}{\ln n}\frac{W(n)+\ln W(n)}{W(n)}\(\frac1{\ln W(n)}+1\)\\
	&\leq  1+\e^{1/5}\frac{\ln \ln n}{\ln n}\(1+\frac{\ln 5}{5}\)\(\frac1{\ln 5}+1\)\\
	&\leq  1+3\,\frac{\ln \ln n}{\ln n}.
	\end{align*}
	Applying this to \eqref{eq:aux15} we get \eqref{eq:expl2}.

We turn our attention to the first inequality in \eqref{eq:expl1}.
	Using \eqref{eq:aux14} and the inequality  $1+x\leq \e^x, x\in\R$, we get for $W(n)\geq1$ (i.e. $n\geq3$)
	\begin{align*}
	E_n^*\geq\[\frac{n}{\e\ln n}{\e^{\frac1{W(n)}-\frac1n-\frac1{2n}W(n)}}\]^n.
	\end{align*}
	Consequently,  Proposition \ref{prop:main2} and the inequality $ 1-x\geq \e^{-2x}$, $x\in [0, 1/2]$, give us
	\begin{align*}B_n\geq \(1-\frac15 \frac{\ln n}{n}\)E_n^*\geq \[\frac{n}{\e\ln n}\exp\({\frac1{W(n)}-\frac1n-\frac{W(n)}{2n}-\frac25\frac{\ln n}{n^2}}\)\]^n.
	\end{align*}
	Furthermore, by $W(x) \leq \ln(x)$ for $x \geq \e$, we get for $n\geq 3$
	\begin{align*}
		&{\frac1{W(n)}-\frac1n-\frac{W(n)}{2n}-\frac25\frac{\ln n}{n^2}}\\
		&=\frac1{W(n)}\(1-\frac{W(n)}n-\frac{(W(n))^2}{2n}-\frac25\frac{W(n)\ln n}{n^2}\)\\
		&\geq \frac1{\ln n}\(1-\frac{\ln n}n-\frac{(\ln n)^2}{2n}-\frac25\(\frac{ n\ln n}{n}\)^2\)\\
		&\geq \frac1{\ln n}\(1-\max_{x\geq1}\frac{\ln x}x-\frac12\max_{y\geq1}\frac{(\ln y)^2}{y}-\frac25\(\max_{z\geq1}\frac{\ln z}z\)^2\)\\
		&\geq \frac1{\ln n}\(1-\frac1{\e}-\frac12\frac{4}{\e^2}-\frac25\frac1{\e^2}\)\\
		&= \frac1{\ln n}\(1-\frac{5\e+12}{5\e^2}\)\geq 0,
	\end{align*}
	which proves the lower bound in the assertion for $n \geq 3$. Numerical verification of the upper bound  in \eqref{eq:expl1} for $2 \leq n \leq 42$, the  lower one for $n = 2$ and the bound \eqref{eq:expl2} for $n\leq742$ ends the proof.
\end{proof}

\begin{remark}
One can see from the proof that considering $n$ large enough, the constant $3/4$ might be replaces by any other greater than $\frac1{\e}$. Nevertheless, $3/4$ is very close to the optimal, when considering all $n\geq2$, which might be observed for $n=10$. Indeed, it holds $B_{10}=115975$ while $(0.739\cdot10/\ln 10)^{10}\approx 115954$.

The bound \eqref{eq:expl2} shows that in \eqref{eq:expl1} the lower one is more optimal.  Namely, one can deduce that 
$$\frac{B_n^{1/n}}{\frac1\e\frac{n}{\ln n}}\stackrel{n\rightarrow\infty}{\longrightarrow}1,$$
which, without convergence rates, has been already observed in \cite{BT}.
\end{remark}

%
%
%
%

\section{Ratios of subsequent Bell numbers}

In order to find the asymptotics of the ratio $B_{n}/B_{n-1}$ we start with analogous problem for the sequence $E_n$ approximating the Bell numbers.
\begin{lemma}\label{lem:E/E} For $n\geq1$ we have
$$-\frac1{W(n)}\leq \frac{E_{n}}{E_{n-1}}-\e^{W(n)}\leq 0.$$
\end{lemma}
\begin{proof}
For $n\leq5$ we verify the assertion numerically. Assume therefore $n\geq6$. From \eqref{eq:aux2} we have
\begin{align*}
E_n=\frac{n!\exp\({\int_{0}^{n+1}W(x)dx-(n+1)\[\ln(n+1)-1\]-\tfrac12W(n+1)}\)}{\sqrt{2\pi\(1+\frac1{W(n+1)}\)}}.
\end{align*}
This implies
\begin{align*}
&\frac{E_{n}}{E_{n-1}}-\e^{W(n)}\\
&=\e^{W(n)}\[\frac{\exp\({\int_{n}^{n+1}W(x)dx-\tfrac12\[W(n+1)+W(n)\]-(n+1)\ln(1+\frac1{n})+1}\)}{\sqrt{\({1+\frac1{W(n)}}\)^{-1}\({1+\frac1{W(n+1)}}\)}}-1\].
\end{align*}
Monotonicity of $W$ and the inequality \eqref{eq:aux7} give us
\begin{align*}
1\leq \sqrt{\frac{1+\frac1{W(n)}}{1+\frac1{W(n+1)}}}&=\sqrt{1+\frac{{W(n+1)}-{W(n)}}{W(n)W(n+1)\(1+\frac1{W(n+1)}\)}}\\
&\leq \sqrt{1+\frac{{1}}{n(W(n))^2}}\leq {1+\frac{{1}}{2n(W(n))^2}}.
\end{align*}
Furthermore, by the mean value theorem and monotonicity of $W$ and $W'$, 
\begin{align*}
0\leq &\int_{n}^{n+1}W(x)dx-\tfrac12\[W(n+1)+W(n)\]\\
&=\int_{n}^{n+1/2}W(x)-W(n)dx+\int_{n+1/2}^{n+1}W(x)-W(n+1)dx\\
&\leq \frac14W'(n)-\frac14W'(n+1)\leq \frac14\max_{x\in[n,n+1]}W''(x)\\
&=\frac14\max_{x\in[n,n+1]}\frac{1+\frac{1}{x+W(x)}}{(x+W(x))^2}\leq\frac{1+\frac1{n}}{4n^2}\leq\frac12\frac{1}{n^2}.
\end{align*}
Using this and the inequalities $x\geq \ln(x+1)\geq x-\tfrac12x^2$, we  get
\begin{align*}
\e^{-1/n}&\leq \exp\({\int_{n}^{n+1}W(x)dx-\tfrac12\[W(n+1)+W(n)\]-(n+1)\ln(1+\frac1{n})+1}\)\\
&\leq \exp\(\frac12\frac{1}{n^2}-\frac{1}{n}+\frac{n+1}{2n^2}\)= \exp\(-\frac1{2n}\Big(1-\frac2{n}\Big)\)\\
&\leq \e^{-1/3n}\leq 1-\frac1{3n}+\frac{1}{18n^2}\leq 1-\frac{1}{4n},
\end{align*}
where we took advantage of the assumption $n\geq6$.
Consequently, we obtain
\begin{align*}
\frac{E_{n}}{E_{n-1}}-\e^{W(n)}&\geq \e^{W(n)}\(\e^{-1/n}-1\)\geq- \e^{W(n)}\frac1{n}=-\frac1{W(n)}
\end{align*}
and since $W(n)\geq W(6)\geq \sqrt{2}$,
\begin{align*}
\frac{E_{n}}{E_{n-1}}-\e^{W(n)}&\leq \e^{W(n)}\[\(1+\frac{{1}}{2n(W(n))^2}\)\(1-\frac{1}{4n}\)-1\]\\
&=-\frac{1}{W(n)}\[\frac14-\frac1{2(W(n))^2}+\frac{1}{8n(W(n))^2}\]\leq0,
\end{align*}
which ends the proof.
\end{proof}

We are now in position to formulate and prove the main result of this section.
\begin{theorem}\label{thm:B/B} For $n\geq1$ it holds 
$$\left|\frac{B_{n}}{B_{n-1}}-\e^{W(n)}\right|\leq  \frac87\frac{1}{W(n)}.$$ 
\end{theorem}
\begin{proof}
By triangle inequality we have
\begin{align}\label{eq:bb0}
\left|\frac{B_{n}}{B_{n-1}}-\e^{W(n)}\right|\leq \left|\frac{B_{n}}{B_{n-1}}-\frac{E_{n}q_{n}}{E_{n-1}q_{n-1}}\right|+\left|\frac{E_{n}q_{n}}{E_{n-1}q_{n-1}}-\frac{E_{n}}{E_{n-1}}\right|+\left|\frac{E_{n}}{E_{n-1}}-\e^{W(n)}\right|.
\end{align}
The last term is already bounded in Lemma \ref{lem:E/E}. 
Let us deal with the other two. First, we rewrite
\begin{align*}
\left|\frac{B_{n}}{B_{n-1}}-\frac{E_{n}q_{n}}{E_{n-1}q_{n-1}}\right|&= \left|\frac{B_{n}-E_{n}q_{n}}{B_{n-1}}+E_{n}q_{n}\(\frac{E_{n-1}q_{n-1}-B_{n-1}}{B_{n-1}E_{n-1}q_{n-1}}\)\right|\\
&= \left|\frac{E_{n}}{E_{n-1}}\frac{\frac{B_{n}}{E_{n}}-q_{n}}{\frac{B_{n-1}}{E_{n-1}}-1+1}+\frac{E_{n}}{E_{n-1}}\frac{q_{n}}{q_{n-1}}\(\frac{q_{n-1}-\frac{B_{n-1}}{E_{n-1}}}{\frac{B_{n-1}}{E_{n-1}}-1+1}\)\right|\\
&\leq \frac{\frac{E_{n}}{E_{n-1}}}{1-|\frac{B_{n-1}}{E_{n-1}}-1|} \[\left|{\frac{B_{n}}{E_{n}}-q_{n}}\right|+\(1+\frac{|q_{n}-q_{n-1}|}{q_{n-1}}\)\left|{q_{n-1}-\frac{B_{n-1}}{E_{n-1}}}\right|\].
\end{align*}
Next, for $n\geq \lceil5\e^5\rceil=743$ we apply Lemma \ref{lem:E/E}, Proposition \ref{prop:main}, the inequalities \eqref{eq:q0}, \eqref{eq:q2} and Theorem \ref{thm:master}, which leads to 
\begin{align}\nonumber
\left|\frac{B_{n}}{B_{n-1}}-\frac{E_{n}q_{n}}{E_{n-1}q_{n-1}}\right|&\leq \frac{\e^{W(n)}}{1-\tfrac1{11}\e^{-W(n)}}\(2+\frac{\e^{-W(n)}}{10n(1-\e^{-W(n)}/12)}\)1.6\e^{-2W(n)}\\\label{eq:bb1}
&\leq \frac{1}{1-\tfrac1{11}\e^{-5}}\(2+\frac{\e^{-5}}{10\cdot5\e^5(1-\e^{-5}/12)}\)1.6\e^{-W(n)}\leq 4 \e^{-W(n)}.
\end{align}
We bound the second term of the right-hand side in \eqref{eq:bb0}  using Lemma \ref{lem:E/E} and the inequalities  \eqref{eq:q0} and \eqref{eq:q2} as follows
\begin{align}\nonumber
\left|\frac{E_{n}q_{n}}{E_{n-1}q_{n-1}}-\frac{E_{n}}{E_{n-1}}\right|&=\frac{E_{n}}{E_{n-1}}\frac{1}{q_{n-1}}\left|q_{n}-q_{n-1}\right|\leq  \frac{1}{1-\e^{-W(n)}}\frac{1}{10n}\\\label{eq:bb2}
&\leq  \frac{1}{1-\e^{-5}}\frac{1}{10n}\leq\frac1{9n}.
\end{align}
Thus, applying \eqref{eq:bb1}, \eqref{eq:bb2} and Lemma \ref{lem:E/E} to \eqref{eq:bb0} we arrive at
\begin{align*}
\left|\frac{B_{n}}{B_{n-1}}-\e^{W(n)}\right|&\leq 4\e^{-W(n)}+\frac1{9n}+\frac{1}{W(n)}\\
&=\frac{1}{W(n)}\(\frac{4W(n)}{\e^{W(n)}}+\frac1{9\e^{W(n)}}+1\)=\frac{1}{W(n)}\(\frac{20}{\e^{5}}+\frac1{9\e^5}+1\)\leq \frac87\frac{1}{W(n)}.
\end{align*}
This completes the proof for $n\geq743$. The rest has been  verified  numerically.
\end{proof}

	\section* {Acknowledgements}
The authors would like to thank Rados{\l}aw Serafin for numerical assistance during initial research.

\section{Appendix} 
This section contains a code in Python of a program that complements several proofs in this article by confirming the assertions for some initial indices. \\

\begin{python}
	import math
	from scipy.special import lambertw
	from decimal import Decimal #To handle really big numbers
	
	#Definition of Lambert W function for simplicity
	def W(n):
		return lambertw(n).real

	#Definition of Bell numbers
	bell_numbers_dict = {}
	bell_numbers_dict[0] = 1
	
	def bell_number(n):
		if n in bell_numbers_dict.keys():
			return bell_numbers_dict[n]
		result = 0
		for i in range(n):
			result += math.comb(n-1, i) * bell_number(i)
		bell_numbers_dict[n] = result
		return result

	#Definition of approximations E_n and E*_n
	def E(n):
		R = Decimal(W(n+1))
		numerator = math.factorial(n) * (R.exp()-1).exp()
		denominator = R**n * (2 * Decimal(math.pi) * (n+1) * (R+1)).sqrt()
		return numerator/denominator

	def E_star(n):
		numerator = (Decimal(W(n)).exp() + Decimal(n*W(n)) - (n+1)).exp()
		denominator = Decimal(math.sqrt(1 + W(n)))
		return numerator/denominator
	
	##############################################
	####### Calculations for THEOREM 4.1 #########
	##############################################
	
	def bounds_4_1(n):
		'''Function returns a tupple
		(lower bound, upper bound)
		for B_n proposed by theorem 4.1'''
		R = W(n+1)
		outer = 1.6*math.exp(-2*R)
		inner = 1 - math.exp(-R)/12 * (1 - 3/(2*R) -
		10/R**2 - 9/R**3 + 1/R**4)/(1 + 1/R)**3
	
		#Theorem 4.1 states that |Bn/En - inner| <= outer,
		#or equivalently En*(inner-outer) <= Bn <= En*(inner+outer)
		
		lower_bound = E(n) * Decimal(inner - outer)
		upper_bound = E(n) * Decimal(inner + outer)
		return lower_bound, upper_bound

	print(" ")
	print("THEOREM 4.1")
	print( "n ln(lower bound) ln(B_n) ln(upper bound) OK/WRONG")
	print(" ")
	
	for n in range(1, 743):
		ln_bell = math.log(bell_number(n))
		lower, upper = bounds_4_1(n)
		ln_lower, ln_upper = lower.ln(), upper.ln()
		if ln_bell < ln_lower or ln_bell > ln_upper:
			flag = "WRONG"
		else:
			flag = "OK"
		print(f"{n}, {ln_lower:.6f}, {ln_bell:.6f}, {ln_upper:.6f}, {flag}")
	
	##################################################
	####### Calculations for PROPOSITION 4.3 #########
	##################################################
	
	#Calculations are split into two parts.
	#Part 1 refers to bound |Bn/En - 1| <= exp(-W(n+1))/11
	#Part 2 refers to bound Bn <= En
	
	def bounds_4_3(n):
		'''Function returns a tupple
		(lower bound, upper bound)
		for B_n stated by proposition 4.3'''
		
		#Prop 4.3, stating that |Bn/En - 1| <= exp(-W(n+1))/11
		#is equivalent to En*(1-exp(-W(n+1))/11)<=Bn<=En*(1+exp(-W(n+1))/11)
		
		lower_bound = E(n) * Decimal(1-math.exp(-W(n+1))/11)
		upper_bound = E(n) * Decimal(1+math.exp(-W(n+1))/11)
		return lower_bound, upper_bound
	
	print(" ")
	print("PROPOSITION 4.3 part 1")
	print( "n ln(lower bound) ln(B_n) ln(upper bound) OK/WRONG")
	print(" ")
	
	for n in range(1, 743):
		ln_bell = math.log(bell_number(n))
		lower, upper = bounds_4_3(n)
		ln_lower, ln_upper = lower.ln(), upper.ln()
		if ln_bell < ln_lower or ln_bell > ln_upper:
			flag = "WRONG"
		else:
			flag = "OK"
		print(f"{n}, {ln_lower:.6f}, {ln_bell:.6f}, {ln_upper:.6f}, {flag}")

	print(" ")
	print("PROPOSITION 4.3 part 2")
	print( "n ln(B_n) ln(E_n) OK/WRONG")
	print(" ")
	
	for n in range(1, 743):
		ln_bell = math.log(bell_number(n))
		ln_E = E(n).ln()
		if ln_bell > ln_E:
			flag = "WRONG"
		else:
			flag = "OK"
		print(f"{n},  {ln_bell:.6f},  {ln_E:.6f},  {flag}")
	
	##############################################
	####### Calculations for Theorem 4.6 #########
	##############################################
	
	def bounds_4_6(n):
		'''Function returns a tupple
		(lower bound, upper bound)
		for B_n proposed by theorem 4.6'''
		lower_bound = Decimal(1 - 1/5 * math.log(n)/n) * E_star(n)
		upper_bound = E_star(n)
		return lower_bound, upper_bound
	
	print(" ")
	print("Theorem 4.6")
	print( "n ln(lower bound) ln(B_n) ln(upper bound) OK/WRONG")
	print(" ")
	
	for n in range(1, 311):
		ln_bell = math.log(bell_number(n))
		lower, upper = bounds_4_6(n)
		ln_lower, ln_upper = lower.ln(), upper.ln()
		if ln_bell < ln_lower or ln_bell > ln_upper:
			flag = "WRONG"
		else:
			flag = "OK"
		print(f"{n}, {ln_lower:.6f}, {ln_bell:.6f}, {ln_upper:.6f}, {flag}")

	##############################################
	###### Calculations for Proposition 4.7 ######
	##############################################
	
	#We split calculations into two parts,
	#corresponding to inequalities (4.3) and (4.4) respectively
	
	print(" ")
	print("Proposition 4.7 first part")
	print( "n ln(lower bound) ln(B_n) ln(upper bound) OK/WRONG")
	print(" ")
	
	for n in range(2, 43):
		ln_bell = math.log(bell_number(n))
		ln_lower = n*math.log(1/math.e * n/math.log(n))
		ln_upper = n*math.log(3/4 * n/math.log(n))
		if ln_bell < ln_lower or ln_bell > ln_upper:
			flag = "WRONG"
		else:
			flag = "OK"
		print(f"{n}, {ln_lower:.6f}, {ln_bell:.6f}, {ln_upper}, {flag}")

	print(" ")
	print("Theorem 4.7 second part")
	print( "n ln(B_n) ln(upper bound) OK/WRONG")
	print(" ")
	
	for n in range(3, 743):
		ln_bell = math.log(bell_number(n))
		ln_upper = n*math.log(n/(math.e*math.log(n))*(1 +	3*math.log(math.log(n))/math.log(n)))
		if ln_bell > ln_upper:
			flag = "WRONG"
		else:
			flag = "OK"
		print(f"{n},  {ln_bell:.6f},  {ln_upper:.6f},  {flag}")

	##############################################
	######## Calculations for Lemma 5.1 ##########
	##############################################
	
	print(" ")
	print("Lemma 5.1")
	print( "n lower bound middle upper bound OK/WRONG")
	print(" ")
	
	for n in range(1, 6):
		upper = 0
		middle = E(n+1)/E(n) - Decimal(W(n+1)).exp()
		lower = Decimal(-1/W(n))
		if middle < lower or middle > upper:
			flag = "WRONG"
		else:
			flag = "OK"
		print(f"{n}, {lower:.6f}, {middle:.6f}, {upper:.6f}, {flag}")

	##############################################
	####### Calculations for Theorem 5.2 #########
	##############################################
	
	def bounds_5_2(n):
		'''Function returns a tupple
		(lower bound, upper bound)
		for B_n proposed by theorem 5.2'''
		
		#Theorem 5.2, stating that |B_n/B_{n-1} - e^{W(n)}| <= 8/7 1/W(n)
		#is equivalent to
		#e^{W(n)} - 8/7 1/W(n) <= B_n/B_{n-1} <= e^{W(n)} + 8/7 1/W(n)
		
		lower_bound = math.exp(W(n)) - 8/7 * 1/W(n)
		upper_bound = math.exp(W(n)) + 8/7 * 1/W(n)
		return lower_bound, upper_bound
	
	print(" ")
	print("Theorem 5.2")
	print( "n lower bound B_n/B_{n-1} upper bound OK/WRONG")
	print(" ")
	
	for n in range(1, 743):
		bell_ratio = bell_number(n)/bell_number(n-1)
		lower, upper = bounds_5_2(n)
		if bell_ratio < lower or bell_ratio > upper:
			flag = "WRONG"
		else:
			flag = "OK"
		print(f"{n}, {lower:.6f}, {bell_ratio:.6f}, {upper:.6f}, {flag}")
\end{python}

 \bibliography{bibliography}    

\begin{thebibliography}{10}

\bibitem{AAR}
A.-M. Acu, J.~A. Adell, and I.~Ra\c{s}a.
\newblock Explicit upper bounds for {T}ouchard polynomials and {B}ell numbers.
\newblock {\em Acta Math. Hungar.}, 172(1):255--263, 2024.

\bibitem{BT}
Daniel Berend and Tamir Tassa.
\newblock Improved bounds on {B}ell numbers and on moments of sums of random
  variables.
\newblock {\em Probab. Math. Statist.}, 30(2):185--205, 2010.

\bibitem{BC}
Jonathan~M. Borwein and O-Yeat Chan.
\newblock Uniform bounds for the complementary incomplete gamma function.
\newblock {\em Math. Inequal. Appl.}, 12(1):115--121, 2009.

\bibitem{CGHJK}
R.~M. Corless, G.~H. Gonnet, D.~E.~G. Hare, D.~J. Jeffrey, and D.~E. Knuth.
\newblock On the {L}ambert {$W$} function.
\newblock {\em Adv. Comput. Math.}, 5(4):329--359, 1996.

\bibitem{B}
N.~G. de~Bruijn.
\newblock {\em Asymptotic Methods in Analysis}.
\newblock Dover Publications, Dover, New York, NY, 1958.

\bibitem{D}
G.~Dobi{\'n}ski.
\newblock Summierung der reihe $\sum n^m/m!$ f{\"u}r $m = 1, 2, 3, 4, 5,... $.
\newblock {\em Grunert Archiv (Arch. Math. Phys.)}, 61:333--336, 1877.

\bibitem{FS}
Philippe Flajolet and Robert Sedgewick.
\newblock {\em Analytic combinatorics}.
\newblock Cambridge University Press, Cambridge, 2009.

\bibitem{GKP}
Ronald~L. Graham, Donald~E. Knuth, and Oren Patashnik.
\newblock {\em Concrete mathematics}.
\newblock Addison-Wesley Publishing Company, Reading, MA, second edition, 1994.
\newblock A foundation for computer science.

\bibitem{K}
Donald~E. Knuth.
\newblock {\em The art of computer programming. {V}ol. 4{A}. {C}ombinatorial
  algorithms. {P}art 1}.
\newblock Addison-Wesley, Upper Saddle River, NJ, 2011.

\bibitem{L}
L\'{a}szl\'{o} Lov\'{a}sz.
\newblock {\em Combinatorial problems and exercises}.
\newblock North-Holland Publishing Co., Amsterdam, second edition, 1993.

\bibitem{MW}
Leo Moser and Max Wyman.
\newblock An asymptotic formula for the {B}ell numbers.
\newblock {\em Trans. Roy. Soc. Canada Sect. III}, 49:49--54, 1955.

\bibitem{O}
A.~M. Odlyzko.
\newblock Search for the maximum of a random walk.
\newblock In {\em Proceedings of the {S}ixth {I}nternational {S}eminar on
  {R}andom {G}raphs and {P}robabilistic {M}ethods in {C}ombinatorics and
  {C}omputer {S}cience, ``{R}andom {G}raphs '93'' ({P}ozna\'{n}, 1993)},
  volume~6, pages 275--295, 1995.

\bibitem{TE}
E.~G. Tsylova and E.~Ya. Ekgauz.
\newblock Using probabilistic models to study the asymptotic behavior of {B}ell
  numbers.
\newblock {\em J. Math. Sci. (N.Y.)}, 221(4):609--615, 2017.

\end{thebibliography}
\bibliographystyle{plain}

\end{document}